\documentclass[12pt,oneside]{article}
\newcommand{\Path }{}
\usepackage{\Path Paper_style}
\usepackage{\Path Keywords}
\usepackage{\Path Environments}
\usepackage{tikz-cd}

\newcommand{\dimM}{n}
\newcommand{\dimG}{d}

\newcommand{\LieG}{G}
\newcommand{\LieA}{\mathfrak{g}}
\newcommand{\Trans}[1]{\left( l_{#1}\right)}
\newcommand{\metric}{\tau}
\newcommand{\IsoA}{\mathcal{I}}
\newcommand{\IsoB}{\Psi}

\begin{document}
\title{\vspace{-2cm}Lie group valued Koopman eigenfunctions}
\author{Suddhasattwa Das\footnotemark[1]}
\footnotetext[1]{Department of Mathematics and Statistics, Texas Tech University, Texas, USA}
\date{\today}
\maketitle

\begin{abstract} Every continuous-time flow on a topological space has associated to it a Koopman operator, which operates by time-shifts on various spaces of functions, such as $C^r$, $L^2$, or functions of bounded variation. An eigenfunction of the vector field (and thus for the Koopman operator) can be viewed as an $S^1$-valued function, which also plays the role of a semiconjugacy to a rigid rotation on $S^1$. This notion of Koopman eigenfunctions will be generalized to Lie-group valued eigenfunctions, and we will discuss the dynamical aspects of these functions. One of the tools that will be developed to aid the discussion, is a concept of exterior derivative for Lie group valued functions, which generalizes the notion of the differential $df$ of a real valued function $f$. The extended notion of Koopman eigenfunctions utilizes a geometric property of usual eigenfunctions. We show that the generalization in a geometric sense can be used to reveal fundamental properties of usual Koopman eigenfunctions, such as their behavior under time-rescaling, and as submersions.
\\\emph{AMSC code} : 22E15, 22E60, 37C85, 37C15
\end{abstract}

\section{Introduction} \label{sec:intro}

A dynamical system is given by a map $f:M\to M$ on a manifold $M$, or a flow $\Phi^t : M\to M$. Dynamical systems theory addresses various questions about these maps / flows, such as the topology of invariant sets, nature of orbits and invariant measures. The Koopman operator framework [see Section~\ref{sec:dynamics}] studies the dynamics induced on observables / measurements, and provide an operator theoretic reformulation of many questions in Dynamical systems theory. 

One of the key objects of studying the Koopman operator are its eigenfunctions for various choices of function spaces. Koopman eigenfunctions can be directly interpreted as rotational dynamics factored into the nonlinear dynamics, see for example \eqref{eqn:def:koop_eigen}, or their use in \cite{Cai_equicont_2020, DasJim2017_SuperC,  DSSY2017_QQ, LevnajicMezic2010}. While the most common choice of space is $L^2(\mu)$, the space of square-integrable functions with respect to a dynamically invariant measure $\mu$, there has been several generalizations of Koopman eigenfunctions in various directions. For example, Klus et. al. introduced a tensorial reformulation of the Koopman eigen-equation \cite{KlusEtAl_tensor_2018} for more robust study of nonlinear systems; Koopman eigenfunctions can be interpreted as measures of coherence  \citep[e.g][]{LopesEtAl_2008, FroylandPadberg09}; as stationary quantum modes \cite{FibichKlein2011}; as Sobolev functions measuring the rate of dissipation in mixing systems \cite{FengIyer2019}; and as the modulating component in signals from \emph{quasiperiodically driven systems} \citep[see][]{DasEtAl_Alfaya_ACC_2021}.

The goal of this paper is to study yet another aspect of Koopman eigenfunctions, the partitions they induce. The partitions induced by Koopman eigenfunctions have been used for ergodic approximations in some systems \citep[see][]{Murray_partition_2004}. We shall look at the geometry of these partitions. One of our main results  is \\

\emph{ \textbf{Submersion theorem.} [Theorem~\ref{thm:submersion}] A collection of $m$ dynamically independent, $C^1$ Koopman eigenfunctions lead to a submersion $\pi:= (z_1, \ldots, z_m)$ into the $m$-dimensional torus $\mathbb{T}^m$.}

\ 

The notion of Koopman eigenfunctions and dynamic independence will be made precise in Section~\ref{sec:dynamics}. The above result is a combination of differential, geometric, and dynamical properties. To aid the proof, we develop the language of \emph{Lie-group valued Koopman eigenfunctions}, one of the main contributions of this paper. An important ingredient of this analysis is an extension of the concept of differential of a function, from real valued functions to Lie group valued functions, in Section~\ref{sec:Lie}. Most of our analysis is global and coordinate free. We revisit our definitions in Section~\ref{sec:local_lift} with a coordinate based approach to reestablish the similarities of our notions with the usual notion of differential of a real valued functions.

\section{Differential for Lie group valued functions} \label{sec:Lie}

Given a $C^1$ function $f:M\to\real$ on a manifold $M$, the differential of $f$ is a 1-form $df$, i.e., an $\real$-linear, $C^0(M)$ valued function on the space $\Gamma^1 (M)$ of $C^1$ vector fields. The action of $df$ on a vector field $V$ coincides with the action of $V$ on $f$, namely, $V(f)$. Moreover, by exploiting the triviality of the tangent bundle of the real line $\real$, one can also express this as the push forward of the vector field under $f$, namely, $f_* V\in T\real$. One of the objectives in this paper is to extend the notion of the differential to Lie group valued functions $z:M\to \LieG$, where $\LieG$ is some Lie group with Lie algebra $\LieA$, as is usual notation in Lie theorey \citep[e.g.][]{Gorbatsevich1994, GallierQuaintance2020}. In the extended definition, $dz$ will be an $\real$-linear, $C^0(M;\LieA)$-valued function on $\Gamma^1 M$, where $C^0(M;\LieA)$ denotes the set of all $C^1$-smooth mappings from $M$ into $\LieA$.

Lie group valued functions occur in the setting of dynamical systems as \emph{Koopman eigenfunctions}, which will be defined later. Koopman eigenfunctions are $\cmplx$-valued functions which evolve with a periodicity under the dynamics, and a collection of $d$ Koopman eigenfunctions can be viewed as a mapping into the $d$-dimensional torus $\TorusD{d}$. The second objective of this paper is introducing a generalized notion of a Lie-group valued Koopman eigenfunction $z$ in terms of its differential $dz$. It will be shown that analogous to the $\cmplx$-valued case, a $\LieG$-valued Koopman eigenfunction factors the dynamics into a flow on $\LieG$. The following will be the standing notations and assumptions.

\begin{Assumption}\label{A:1}
$M$ is a $C^1$ $\dimM$-dimensional manifold. $\LieG$ is a $\dimG$-dimensional Lie group, with Lie algebra $\LieA$.
\end{Assumption}

The identity element of $\LieG$ will be denoted as $e$. In a Lie group, for every $g\in G$, the left multiplication by $g$ which will be denoted as $\Trans{g}$, is a $C^\infty$ diffeomorphism of $\LieG$. One has,
\[\Trans{g}:\LieG\to \LieG; \quad \Trans{g}:h\mapsto gh; \quad \Trans{g}_*|_h : T_h\LieG \to T_{gh}\LieG; \quad \forall h\in \LieG.\]
Let $z:M\to G$ be a $C^1$ function. We will now proceed to define $dz$ by extending the definition of the exterior derivative $df$ of an $\real$-valued function $f$. Recall that $df$ can be defined as $\proj_2\circ f_*$, with $\proj_2:T\real\to\real$ being the projection onto the second coordinate. This is shown in the commuting diagram below.
\begin{equation}\label{eqn:def:df}
\begin{tikzcd}[column sep = large]
M \arrow{d}{f} &TM \arrow{l}{} \arrow{d}{f_*} \arrow[bend left = 30]{dr}{df} &\ \\
\real &T\real \cong \real\times\real \arrow[swap]{l}{\proj_1} \arrow{r}{\proj_2} &\real 
\end{tikzcd}
\end{equation}
One can easily extend this definition for an $\real^d$ valued function $f:M\to\real^d$, by computing the differential of each of the components. The projection $\proj_2$, which is an important tool in the definition \eqref{eqn:def:df}, will also be well defined in the case of Lie groups too, since Lie groups have trivial tangent bundles. More explicitly, we will use the trivialization $\LieG\times\LieA \cong T\LieG$, with the bundle isomorphism given by 
\begin{equation}\label{eqn:Lie_bundle}
\IsoA:\LieG\times\LieA \to T\LieG \quad := \quad (g,v) \mapsto \Trans{g}_* |_e  v.
\end{equation}
We can now define $dz:\Gamma^1(M)\to C^0(M;\LieA)$ as $dz = \proj_2 \circ z_*$, as shown by the dashed arrow in the commutative diagram below. 
\begin{equation}\label{eqn:def:dz}
\begin{tikzcd}[column sep = large]
M \arrow{d}{z} &TM \arrow{l}{} \arrow{d}{z_*} \arrow[dashed]{r}{dz} & \LieA \\
G &TG \arrow{r}{\IsoA}[swap]{\cong}  \arrow[swap]{l}{\proj_1} &G\times\LieA \arrow{u}{\proj_2}  
\end{tikzcd}.
\end{equation}

Theorem \ref{thm:dz_basic} below gives another equivalent definition for $dz$, and states properties of $dz$ which are similar to those for the usual exterior derivative.

\begin{theorem}\label{thm:dz_basic}
Let Assumption \ref{A:1} hold, and $z:M\to G$ be a $C^1$ function. Then 
\begin{enumerate}[(i)]
\item $dz$ defined as in \eqref{eqn:def:dz} is a $\LieA\cong\real^\dimG$ valued 1-form.
\item At every point $x\in M$, the linear map $dz(x):T_x M\to \LieA$ is explicitly given by
\begin{equation}\label{eqn:def2:dz}
dz(x)(v) = \Trans{z(x)^{-1}}_*|_{z(x)} z_*|_x v, \quad \forall v\in T_x M.
\end{equation}
\item Naturality : Let $\phi:M'\to M$ be a smooth map between manifolds, and $z':M'\to\LieG$ be the map $z'=z\circ\phi$. Then $dz'$ = $dz\circ \phi_*$.
\item Invariance under translation : For every $g\in G$, $d(gz) = d(z)$.
\end{enumerate}
\end{theorem}

Theorem~\ref{thm:dz_basic} is proved in Section~\ref{sec:proof:dz_basic}.
]black{Henceforth, given a vector field $V$ we shall adopt a reuse of notation done for $\real$-valued functions, and denote by $dz(V)$ the map 
\[ x \mapsto dz(V)(x) = dz(V(x)) \] }

\paragraph{Remark} Equation \eqref{eqn:def:dz} defines the tensor $dz$ globally in a coordinate-free manner. Given any $r$-tensor $\alpha$ and any point $x$ on $M$, $\alpha$ restricts to $T_x^r M$ as a linear map $\alpha(x)$. This local definition of $df$ is provided by \eqref{eqn:def2:dz}. The bundle map definition in \eqref{eqn:def:dz} is the same as \eqref{eqn:def:df}, except stated for a general Lie group $\LieG$ instead of $\real$. When $\LieG=\real$, then $\Trans{z(x)^{-1}}_*$ is the identity, so \eqref{eqn:def2:dz} also becomes the usual definition for the differential of $\real$-valued functions. In Claims (iii)-(iv), we continue this analogy by showing that many properties of $df$ carry over to $dz$ as well. 

\paragraph{Gradient of Lie-group valued functions.} Having established a definition of $d\theta$ that is analogous to the differential of an $\real^d$ valued function, one can define the notion of gradient. For this purpose, let $\metric$ be a Riemannian metric on $M$. Then the gradient $\nabla_{\metric} z$  of $z$ in this metric is the dual of the $\LieA$-valued 1-form $dz$. In other words $\nabla_{\metric} z$  will be a collection of $d$ vectors defined through dual action.
\begin{equation}\label{eqn:def:gradient}
\langle \nabla_{\metric} z, W \rangle_\metric := dz(W); \quad \forall W\in\Gamma^1 M.
\end{equation}
Here, $\langle \nabla_{\metric} z, W \rangle_\metric$ is the collection of the inner products of $W$ with the $d$ components of $dz$. Let $0_\LieA$ denote the $0$ element of $\LieA$. Then by the \emph{kernel} of the gradient $\nabla_{\metric} z$ at $x$, we will mean the set of vectors $w\in T_x M$ such that $\langle \nabla_{\metric} z(x), w \rangle_\metric$   = $0_\LieA$. This will be denoted $\ker \nabla_{\metric} z$ .

\begin{theorem}\label{thm:gradient}
Let Assumption \ref{A:1} hold and $\metric$ be a Riemannian metric on $M$. Let $z:M\to\LieG$ be a $C^1$ map, and  $g\in \LieG$ be a regular value of $z$. Thus $N$:= $z^{-1}(g)$ is a codimension-$d$ submanifold. Then at every point $x\in N$, the sub-bundle $TN$ coincides with the kernel of $\nabla_{\metric} z$ (in the $\metric$ - metric).
\end{theorem}

Theorem \ref{thm:gradient} is proved in Section~\ref{sec:proof:gradient}. A Riemannian metric tensor is a useful tool to connect tangent and cotangent spaces. Although it is not the unique tool for this purpose, it provides in addition an intuitive notion of normal direct to a submanifold, as stated in Theorem \ref{thm:gradient}. This notion of a metric tensor along with the gradient is later used in Lemma \ref{lem:Koop_eigen_subm}  to prove that every Koopman eigenfunction with nonzero eigenfrequency is a submersion. 

We have so far extended the geometric notions of differential and gradient of $\real$-valued functions to Lie group valued functions, and the results state that they retain certain analogous properties. These extended notions will now be used in the context of $C^1$ flows on dynamical systems.

\section{Dynamical systems} \label{sec:dynamics}

We now assume the following dynamics occurring on the manifold $M$.

\begin{Assumption}\label{A:2}
$V$ is a $C^1$ non-vanishing vector field on $M$, inducing a flow $\Phi^t : M\to M$. 
\end{Assumption}

\paragraph{The Koopman operator.} Koopman operators \cite{DasGiannakis_delay_2019,DasJim2017_SuperC} act on observables by composition with the flow map, i.e., by time shifts. There are various choices for the space of observables, such as $L^2$, Lipschitz, and functions of bounded variation. Here, we will restrict our attention to $C^1(M)$ : given an observable $f \in C^1(M)$ and time $t\in\real$, $U^t:C^1(M) \to C^1(M)$ is the operator defined as
\[(U^tf):x\mapsto f\left(\Phi^t x\right).\]
$U^t$ is called the Koopman operator associated with the flow, at time $t$. In general, if $\Phi^t$ is a $C^k$ flow for some $k\geq 0$, then $U^t$ maps the space $C^r(M)$ into itself, for every $0\leq r\leq k$. The vector field $V$ acts as a differentiation operation $V:C^1(M) \to C^0(M)$ , defined as
\begin{equation*}\label{eqn:def_gen_flow}
V f:=\lim_{t\to 0} \frac{ 1 }{ t } \left(U^t f - f\right), \quad f \in C^1(M).
\end{equation*}
The family of operators $U^t$ can also be defined on the space of $L^2$ functions with respect to an invariant measure $\mu$ (see for example \cite{DasGiannakis_delay_2019}), in which case, $U^t$ is a 1-parameter unitary group. Such a group has a generator $\hat{V}$ acting on some dense subspace of $L^2(\mu)$, and the action of $\hat{V}$ on $C^1(M)\cap L^2(\mu)$ coincides with that of $V$. The operator theoretic study of a dynamical system is the study of this operator $U^t$ instead of actual structures on the phase space $M$. The primary advantage is that, however nonlinear the underlying system $\Phi^t : M\to M$ is,  $U^t$ is always a bounded linear operator. Thus its dynamics is that of a linear system. On the other hand, instead of the finite dimensional phase space $M$, one has to consider dynamics in an (usually) infinite dimensional space $L^2(\mu)$ .

\paragraph{Koopman eigenfunctions.} A $C^1$ function $z:M\to\cmplx$ is said to be a Koopman eigenfunction with eigenfrequency $\omega$ if for every $x\in M$, every $t\in\real$, $(U^t z)(x) = e^{i\omega t} z(x)$. The operators $U^t$ and $V$ share the same $C^1$ eigenfunctions : 
\begin{equation}\label{eqn:def:koop_eigen}
U^tz=\exp(i\omega t)z \quad \Leftrightarrow \quad Vz=i\omega z.
\end{equation}
Koopman eigenfunctions factor the dynamics onto a rotation on $S^1$ with frequency $\omega$, as shown in the left diagram in \eqref{def:Koop_eigen}. Similarly, $d$ Koopman eigenfunctions $(z_1,\ldots,z_d)$ factor the dynamics into a rotation on $\TorusD{d}$.
\begin{equation}\label{def:Koop_eigen}
\begin{tikzcd}[row sep = small]
M \arrow[swap]{d}{z\ } \arrow{r}{\Phi^t_V} &M \arrow{d}{\ z} \\
S^1 \arrow{r}{R_\omega^t} &S^1
\end{tikzcd}; \quad
R_\omega^t(\theta) \mapsto \theta+t\omega \bmod S^1; \quad
\begin{tikzcd}[row sep = small]
M \arrow[swap]{d}{(z_1,\ldots,z_d)\ } \arrow{r}{\Phi^t_V} &M \arrow{d}{\ (z_1,\ldots,z_d)} \\
\TorusD{d} \arrow{r}{R_\omega^t} &\TorusD{d}
\end{tikzcd}; \quad
R_\omega^t(\theta) \mapsto \theta+t\omega \bmod \TorusD{d}; \quad
\end{equation}
The commutation in \eqref{def:Koop_eigen} holds regardless of whether the $z_1,\ldots, z_d$ are independent or are repeated. We later examine the consequences of these eigenfunctions being independent in Theorem~\ref{thm:submersion}.

In summary Koopman eigenfunctions represent a low dimensional dynamics embedded in the flow, called \emph{quasiperiodic} dynamics. Let $\Disc$ denote the $L^2(\mu)$ closure of the span of the eigenfunctions of $V$, and $\Disc^\bot$ denote its orthogonal complement. This leads to an invariant orthogonal splitting :
\begin{equation} \label{eqn:L2_decomp}
L^2(X,\mu)=\Disc\oplus\Disc^\bot.
\end{equation}
The subspace $\Disc$ always contains the constant functions, which correspond to eigenfrequency $0$. If $d\geq 1$, there is at least one non-trivial Koopman eigenfunction. $\Disc$ is called the \emph{discrete / quasiperiodic } component \citep[see][]{DasGiannakis_RKHS_2018, DasGiannakis_delay_2019} due to their similarities with torus rotations. For example it allows numerically stable forecasting \cite{ZhaoGiannakis2016,BerryEtAl2015}; and has been shown to have excellent convergence properties of ergodic averages \cite{DasJim2017_SuperC,DSSY2017_QQ}. Koopman eigenfunctions, besides their dynamical significance \cite{DasGiannakis_RKHS_2018, DasGiannakis_delay_2019}, have many applications, such as prediction of observables \cite{Giannakis17,DGJ_compactV_2018}; recovery of coherent spatiotemporal patterns \cite{GiannakisEtAl2015,GiannakisDas_tracers_2019}; and information theoretic aspects \cite{DasDimitEnik2020}.  The subspace $\Disc^\bot$ is called the \emph{continuous / chaotic component} as it is the spectral analog of the chaotic component in the dynamics. $\Disc^\bot$ is characterized by slower convergence rates of ergodic averages \cite{LevnajicMezic2010,DSSY_Mes_QuasiP_2016}. So far, studying the action of $V$ on $\Disc^\bot$ has proved to be very challenging. 

\paragraph{The exponential map.} We will now consider a special class of flows on the Lie group, which is based on the \emph{exponential map} of a Lie group. This map, denoted as $\exp:\LieA\to \LieG$ , is one of the features of Lie groups which distinguishes it from ordinary manifolds or topological groups. It provides the connection between the Lie algebra structure of $\LieA$ and the manifold properties of $\LieG$, ]black{although we do not explicitly employ the Lie brackets}. There are many equivalent ways to define the exponential map, the one that will be used here is the following : let $\omega \in \LieA$ be fixed, then there is a vector field $W$ defined as $W(g)$ = $\Trans{g}_*|_e \omega$. $W$ induces a flow $\Phi^t_W$ on $\LieG$, and $\exp(t\omega)$ is defined to be the point $\Phi^t_W e$. In particular, $\exp(\omega)$ := $\Phi^1_W e$. It has the following additional properties.
\begin{enumerate}[(i)]
	\item The vector field $W$ is a translation invariant vector field .
	\item The exponential map $\exp$ maps $0$ into $e$ and is a local diffeomorphism at $0$.
	\item For every $t\in\real$ and $\omega\in\LieA$, the map $\exp^t_\omega : \LieG\to\LieG$ defined as $z\mapsto z \exp \paran{t\omega}$ is a diffeomorphism. 
	\item Moreover, $t\mapsto \exp^t_\omega$ is a flow on $\LieG$. 
\end{enumerate}
Note that we use the notation $\exp( \cdot)$ to denote the exponential map $\exp : \LieA \to \LieG$, and the notation $\exp_\omega^t( \cdot)$ to denote a flow on $\LieG$. Using the exponential map we now extend the notion of eigenvalues in \eqref{eqn:def:koop_eigen}.

\paragraph{Lie-group valued Koopman eigenfunctions.} A function $z:M\to \LieG$ will be called a $\LieG$-valued Koopman eigenfunction with frequency $\omega\in\LieA$ if the following commutative diagram is satisfied.
\begin{equation}\label{def:Koop_G}
\begin{tikzcd}[row sep = small]
M \arrow[swap]{d}{z\ } \arrow{r}{\Phi^t_V} &M \arrow{d}{\ z} \\
\LieG \arrow{r}{\exp^t_\omega} &\LieG
\end{tikzcd}; \quad
\Leftrightarrow \quad
z(x) \exp \paran{t\omega} = z\circ \Phi^t_V(x), \quad \forall t\in\real, \quad \forall x\in M.
\end{equation}
Note that this is an extension of the definition of a Koopman operator in \eqref{def:Koop_eigen}, where both $S^1$ and $\TorusD{d}$ are Lie groups. In the case of a more general Lie group, its Lie algebra plays the role of frequency. Thus the key idea of this extension is to interpret the usual eigenfrequency, a scalar, as a vector in $\LieA$. The next theorem gives an equivalent characterization of a $\LieG$-valued Koopman eigenfunction, in terms of the action of $V$.

\begin{theorem}\label{thm:Koop_equiv}
Let Assumptions \ref{A:1}, \ref{A:2} hold, $z:M\to\LieG$ be a $C^1$ map, and $dz$ be the exterior derivative defined in \eqref{eqn:def2:dz}. Then the following are equivalent.
\\(i) $z$ is a $G$-valued Koopman eigenfunction, and it satisfies \eqref{def:Koop_G} for some $\omega\in\LieA$.
\\(ii) $dz(V)(x) =\omega$ for every $x\in M$.
\end{theorem}

Theorem~\ref{thm:Koop_equiv} is proved in Section~\ref{sec:proof:Koop_equiv}. The proof reveals that this generalization would not have been possible to $G$-valued eigenfunctions, with $G$ being an arbitrary topological group. The Lie algebra structure plays a key role in defining the exponential map and thus the commutation in \eqref{def:Koop_G}. We next look at two consequences of our notions in two important aspects of dynamical systems theory - (i) the dynamical effect of rescaling the vector field; and (ii) the notion of independence of Koopman eigenfunction.

\paragraph{Rescaling vector fields.} Given a positive $C^1$ map $\alpha:M\to (0,\infty)$, one can get a rescaled vector field $\tilde{V} := \alpha V$. At each $x\in M$, the scalar $\alpha(x)$ rescales the vector $V(x)$ to $\tilde{V}(x)$, and thus, the flows induced by $V$ and $\tilde{V}$ will have the same trajectories, but with different speeds along each trajectory. It is well known that rescaling of vector fields can change the spectral properties or spectral type of the Koopman operator. It can also change properties such as mixing and ergodicity, see for example \cite{Kocergin1973,Chacon1966,Parry1972}. Most of these results show that arbitrary flows (continuous or measurable) can be rescaled by arbitrarily small amounts so as to produce a mixing flow, i.e. a flow without Koopman eigenfunctions. The following result is in an opposite direction, it provides necessary and sufficient geometrical conditions under which a submersion into $\LieG$ can be made into a Koopman eigenfunction by rescaling the vector field.

\begin{theorem}\label{thm:rescale_G}
Let Assumptions \ref{A:1}, \ref{A:2} hold and $z:M\to G$ be a non-constant, $C^1$ map. Then the following hold.
\\(i) For every $\alpha\in C^0(M;\real)$, $\alpha\neq 0$ everywhere, $dz(\alpha V)$ = $\alpha dz(V)$.
\\(ii) There is a rescaling of the vector field $V$ which makes $z$ an eigenfunction iff there is a 1-dimensional subspace $L$ of $\LieA$ such that for every $x\in M$, $dz(V)(x)\neq 0$ and lies in $L$. 
\end{theorem}

Theorem \ref{thm:rescale_G} is proved in Section~\ref{sec:proof:rescale_G}. The criterion for $z$ being non-constant was included to exclude the case when $\omega=0$. Note that $dz$ is $\real$-linear by virtue of its construction \eqref{eqn:def:dz}. Part~(i) of the theorem shows that it is also $C^0(M)$ linear. The results of Theorems \ref{thm:dz_basic}, \ref{thm:gradient}, \ref{thm:Koop_equiv} and \ref{thm:rescale_G} will now be summarized for the case when $\LieG = S^1$.

\begin{corollary}\label{corr:rescale:S1}
Let $M$ be a $C^2$ manifold, $V$ a nonzero $C^1$ vector field on $M$ inducing a $C^1$ flow $\Phi^t$. Let the unit circle $S^1$ be identified with the unit circle in the complex plane $\cmplx$. Let $z:M\to S^1$. Then the following hold.
\begin{enumerate}[(i)]
\item $dz$ defined as $\proj_2\circ z_*$, as shown below, is a 1-form.
\begin{equation}\label{eqn:def2:dz_S1}
\begin{tikzcd}[column sep = large]
M \arrow{d}{z} &TM \arrow{l}{} \arrow{d}{z_*} \arrow[bend left = 30]{dr}{dz} &\ \\
S^1 &TS^1 \cong S^1\times\real \arrow[swap]{l}{\proj_1} \arrow{r}{\proj_2} &\real
\end{tikzcd}.
\end{equation}
\item Let $U\subset M$ be an open set and $\theta:U\to\real$ such that for every $x\in U$, $z(x) = \exp(i\theta(x))$ . Then $d\theta$ = $d(z|U)$.
\item Naturality : Let $\phi:M'\to M$ be a smooth map between manifolds, and $z':M'\to S^1$ be the map $z'=z\circ\phi$. Then $dz'$ = $dz\circ \phi_*$.
\item For every $f\in C^1(G;\real)$, $V(f\circ z) = f_* dz(V)$.
\item Suppose that $\gamma\in S^1$ is a regular value of $z$, then $N$:= $z^{-1}(\gamma)$ is a codimension-$1$ submanifold. Fix a metric $\tau$ on $M$. Then at every point $x\in N$, the sub-bundle $TN$ coincides with the kernel of $\nabla_{\metric} z|N$.
\item Let $\zeta:M\to \cmplx$ be a $C^1$ submersion. Then $V$ can be rescaled so that $\zeta$ becomes an eigenfunction iff the following conditions hold
\\(i) $|\zeta|$ is constant everywhere.
\\(ii) $V$ is transversal to the foliation induced by $\zeta$.
\end{enumerate}
\end{corollary} 

Corollary \ref{corr:rescale:S1} is proved in Section~\ref{sec:proof:rescale:S1}. There has been several explorations of the link between foliations and Koopman eigenfunctions \citep[e.g.][]{korda2020optimal, shirasaka2020phase}. Corollary \ref{corr:rescale:S1}~(vi) uncovers one aspect of it.  One of the key aspects in these results is the transversality of the foliations. We next look at an important consideration for Koopman eigenfunctions : \emph{independence}.

\paragraph{Generating eigenfunctions} Suppose that the flow generated by $V$ has an invariant probability measure $\mu$ supported on $M$. The Koopman operator $U^t$ then also acts on $L^2(\mu)$ as a unitary operator. As a result, its eigenvalues are purely imaginary numbers of the form $\iota \omega$. The imaginary component $\omega$ is also called an \emph{eigenfrequency}. Eigenfrequencies form a module over the ring over integers $\integer$, i.e., they are closed under finite integer linear combinations. Given two Koopman eigenfunctions $z_1, z_2$ with eigenfrequencies $\omega_1, \omega_2$, for every pair of integers $a, b\in \integer$, $z_1^a z_2^b$ is a Koopman eigenfunction with eigenfrequency $a\omega_1 + b\omega_2$, as shown below.
\[\begin{split}
    U^t \left( z_1^a z_2^b \right) (x) &= \left( U^t z_1^a (x) \right) \left( U^t z_2^b (x) \right) = \left( z_1 ( \Phi^t x) \right)^a \left( z_2 (\Phi^t x) \right)^b = \left( e^{\omega_1 t} z_1 (x) \right)^a \left( e^{\iota\omega_2 t} z_2 (x) \right)^b \\
    &= e^{\iota(a\omega_1 + b\omega_2) t} \left( z_1^a z_2^b \right) (x).
\end{split}\]
A collection of eigenfrequencies are said to be \emph{independent} if no non-trivial integer linear combination of them is an integer. A collection of eigenfunctions is said to be independent if their corresponding eigenfrequencies are independent. A set of eigenfrequencies is said to be a \emph{generating} set of eigenfrequencies if they are independent and every eigenfrequency is an integer linear combination of these frequencies. If a manifold is finite dimensional, then every generating set of frequencies is finite and has the same size, say $d$. Given such a generating set $\omega_1, \ldots, \omega_d$, every eigenfrequency of $V$ is of the form $\sum_{j=1}^{d} a_j \omega_j$, for some $(a_1,\ldots,a_d)\in\integer^d$. If $d>1$, then the eigenfrequencies are dense on the real line. Given any vector $\vec{a}=(a_1,\ldots,a_d) \in\integer^d$, we denote by $z_{\vec{a}}$ a unit norm eigenfunction with eigenfrequency $\vec{a}\cdot\vec{\omega} = \sum_{j=1}^{d} a_j \omega_j$. 

Theorem \ref{thm:submersion} below seems to be a fundamental result for Koopman theory, but it is being stated and proved here because no formal statement or proof has been found by the authors in the existing literature. 

\begin{theorem}\label{thm:submersion}
Let Assumption \ref{A:1} hold, and $V$ be a $C^1$ vector field inducing an ergodic flow on $M$. Suppose that $z_1,\ldots,z_m$ are $m$, independent, $C^1(M)$ eigenfunctions. Then the map $\pi:= (z_1, \ldots, z_m) :M\to\TorusD{m}$ is a submersion.
\end{theorem}

Theorem~\ref{thm:submersion} is proved in Section~\ref{sec:proof:submersion}. If one fixes a generating set of frequencies $\vec\omega \in \real^d$ and a corresponding eigenfunction $z_{\vec\omega} : M \to \mathbb{T}^d$, then the values of $z_{\vec\omega}$ can be interpreted as the dynamic phase of a point in the state space. Under the conditions of Theorem~\ref{thm:submersion}, the points with equal phase can be shown to be codimension-$d$ submanifolds. This is used in a reconstruction of the dynamics in an ongoing work.

Next in Section \ref{sec:local_lift}, we provide yet another characterization of the differential, in terms of a local representation of $z$ as an $\real^d$ valued function.


\section{Local lifts of Lie-group valued functions} \label{sec:local_lift}

Recall that the real line $\real$ is a covering space for the circle $S^1$, with the exponential map $x\mapsto \exp(i x)$ as the covering map. Any $S^1$ valued map $z$ can be lifted over small neighborhoods to an $\real$-valued map $\bar{\theta}$, and the differential properties of $\bar{\theta}$ coincide with that of $z$. In particular, $d\bar{z} = dz$. This idea will now be extended to general Lie groups.

Since the exponential map $\exp$ is a local diffeomorphism at $0_\LieA$, one can fix a neighborhood $U_0$ of $0_\LieA$ such that $\exp$ is a diffeomorphism of $U_0$ onto its image. This image $U_e$:= $\exp(U_0)$, will be a neighborhood of $e$ in $G$. Let $x\in M$ and $g$ := $z(x)$. Then $U_g$:= $\Trans{g} (U_e)$ is a neighborhood of $g$ in $G$, and $U_x$ := $z^{-1}(U_g)$ is a neighborhood of $x$ in $M$. Let $E_g:U_0\to U_g\subset G$ be the map $v\mapsto l_g\exp(v)$. Then there is a map $\theta : U_x\to U_0\subseteq\real^d$ such that $z = E_g\circ \theta$, as shown below.
\begin{equation}\label{eqn:def:lift}
\begin{tikzcd}[column sep = large]
U_e \arrow{r}{l_g}[swap]{\cong} &U_g \\
U_0 \arrow{u}{\exp}[swap]{\cong} \arrow[swap]{ur}{E_g} &U_x \arrow{l}{\theta} \arrow{u}{z}
\end{tikzcd}.
\end{equation}
This $\LieA$ valued map $\theta$ will be called a lift of $z$ at $x$, under the exponential map. 

We will now consider a trivialization of the tangent bundle of $\LieA$. Since $\LieA \cong \real^d$, $T\LieA$ already has a canonical bundle isomorphism $\IsoA_{\text{can}} : T\LieA \to \LieA\times\LieA$. However, we will consider a different bundle isomorphism, based on the following mapping which maps every $u\in \LieA$ into $e\in\LieG$.
\[u \xrightarrow{\exp} \exp(u) \xrightarrow{\Trans{\exp(u)^{-1}}} e , \quad \forall u\in \LieA. \]
The induced maps between the corresponding tangent bundles leads to the following bundle map $\IsoB$.
\begin{equation}\label{eqn:def:Iso2}
\IsoB : TU_0 \to U_0\times\LieA \quad := (u,w_u) \mapsto \left( u, \Trans{\exp(u)^{-1}}_*|_{\exp(u)} \exp_*|_u w_u \right).
\end{equation}
One has to check that the composition of maps on the right hand side of \eqref{eqn:def:Iso2} makes sense : $\exp : u \mapsto \exp(u)$, so $\exp_*|_u : T_u U_0 \to T_{\exp(u)} \LieG$, and $\Trans{\exp(u)^{-1}} : \exp(u) \mapsto e$, so $\Trans{\exp(u)^{-1}}_*|_{\exp(u)} : T_{\exp(u)} \LieG \to T_e \LieG = \LieA$. Thus their composition is a map from $T_u U_0$ into $\LieA$. Now consider the maps $d\theta$ and $\tilde{d}\theta$ as follows.
\begin{equation}\label{eqn:def:dtheta}
\begin{tikzcd}
U_x \arrow{d}{\theta} &TU_x \arrow{l}{} \arrow{d}{\theta_*} \arrow[bend left = 40]{drr}[swap]{\tilde{d}\theta} &\ &\ \\
U_0 &TU_0 \arrow{r}{\IsoB} \arrow{l}{} &U_0\times\LieA \arrow{r}{\proj_2} &\LieA 
\end{tikzcd}
;\quad 
\begin{tikzcd}
U_x \arrow{d}{\theta} &TU_x \arrow{l}{} \arrow{d}{\theta_*} \arrow[bend left = 40]{drr}[swap]{d\theta} &\ &\ \\
U_0 &TU_0 \arrow{r}{\IsoA_{\text{can}}} \arrow{l}{} &U_0\times\LieA \arrow{r}{\proj_2} &\LieA 
\end{tikzcd}.
\end{equation}
If $\LieG$ is a commutative Lie group, then $\IsoB$ coincides with $\IsoA_{\text{can}}$, and $\tilde{d}\theta$ coincides with $d\theta$. This will be stated and proved as a part of Theorem \ref{thm:local_lift} below, which also connects the differential of a local lift $\theta$ to the differential of $z$, using the map $\IsoB$. 

\begin{theorem}\label{thm:local_lift}
Let Assumption \ref{A:1} hold, and let $z : M\to\LieG$ be a $C^1$ map. Let $\theta$ be a local lift of $z$ at a point $x$, as defined in \eqref{eqn:def:lift}. Then $d(z|U)=\tilde{d} \theta$. If $\LieG$ is a commutative Lie group, then $\IsoB \equiv \IsoA_{\text{can}}$, and $d(z|U) = \tilde{d} \theta = d\theta$.
\end{theorem}

Theorem~\ref{thm:local_lift} is proved in Section~\ref{sec:proof:local_lift}.
This completes the description of our ideas and statement of our results. We end the paper with some discussion and then the proofs of our theorems. 
\section{Conclusions} \label{sec:conclus}

The first step of this treatise was a generalization of the concept of the differential of a scalar valued function to the differential of a Lie group valued function, via \eqref{eqn:def:dz}, \eqref{eqn:def2:dz}. Using this geometric notion, the next step was to define Lie-group valued Koopman eigenfunctions as a semiconjugacy of the original flow to an exponential flow on the Lie group \eqref{def:Koop_eigen}. This was shown equivalent to a differential property in Theorem~\ref{thm:Koop_equiv}, where the notion of the differential was essential. At this point, we explore the significance of these generalizations, and some future directions of research.

Firstly, ]black{ although  Theorem~\ref{thm:submersion} is about the usual notion of complex-valued Koopman eigenfunctions, its proof relies on the extended notions of the differential.  Theorems \ref{thm:gradient} and \ref{thm:Koop_equiv} are indispensable to the proof, and highlights the necessity of the concept of differentials of Lie group valued functions. } 

Secondly, a key interpretation of Koopman eigenfunctions, $\cmplx$-valued or $G$-valued, is as a factor map into a translation flow, as shown in \eqref{def:Koop_eigen}, \eqref{def:Koop_G}. Just as $\cmplx$-valued Koopman eigenfunctions reveal quasi-periodic modes embedded in the dynamical system, a $G$-valued dynamical system reveals Lie-group translation flows embedded within the dynamics. This embedding marks the point where the properties specific to $G$ has a bearing on the original dynamics. Matrix Lie groups such as $SO(n)$ and $\text{Symplectic}(n)$ have a rich geometric structure \citep[e.g.][]{Gallier2001,hall2003lie}. Each warrant an independent study and lies outside the scope of this present work.

Thirdly, due to the simple evolution law \eqref{def:Koop_G}, these eigenfunctions could be used as a basis for control. A common practice in the model of mechanical systems is to model the configuration space of the system as the Lie-group $G = \real^a \times \mathbb{T}^b \times SO(c)$, to account for degrees of freedom in position, angular coordinates, and rotation \citep[e.g.][]{bloch2002symmetric, Fedorov2005discrete, selig2004lie}. Let $y$ represents a point on $G$ driven by the dynamical system on $M$, namely :
\[\frac{d}{dt} x(t) = V(x(t)), \quad     y(t) = z(x(t)) ,\]
with $z$ being a map $M\to G$.
Then if $z$ corresponds to a $G$-valued Koopman eigenfunction, one can derive an explicit formula for the evolution of the state $z$ on the Lie group $G$. An interesting direction of investigation is if a general function $z:M\to G$ could be decomposed along different $G$-valued Koopman eigenfunctions, and then used for control. This would extend the ideas explored in \cite{korda2020optimal}.

Fourthly and finally, Lie-group valued eigenfunctions which correspond to $\omega=0_{\LieA}$ reveal the group of symmetries for the dynamics \citep[e.g.]{hydon2000symmetry, starrett2007solving}. A promising avenue of research is to develop data-driven methods to identify such $G$-valued eigenfunctions. These in turn could be useful to exploit the symmetries and develop efficient numerical integrators \citep[e.g.]{iserles2000lie, hairer2006geometric, liu2013lie}

\section{Proofs of the Theorems} \label{sec:proofs}


\subsection{Proof of Theorem \ref{thm:dz_basic}} \label{sec:proof:dz_basic}

Since $dz$ is the composition of the linear bundle maps $\IsoA$, $z_*$, followed by a projection into $\LieA$. Thus $dz$ is a $\LieA$-valued 1-form, proving Claim~(i). In Claim~(ii), we have to show that the two definitions in \eqref{eqn:def:dz} and \eqref{eqn:def2:dz} are equivalent. The following is the inverse of  the bundle isomorphism in \eqref{eqn:Lie_bundle}.
\[ T\LieG \cong \LieG\times\LieA; \quad v_g \mapsto \Trans{g^{-1}}_*|_{g} v_g; \quad \forall g\in\LieG, \forall v_g\in T_g \LieG. \]
Inserting this into \eqref{eqn:def:dz} produces \eqref{eqn:def2:dz}.

To prove Claim~(iii), we will show that $dz'(x')(v) = (dz \circ\phi_*)(x') v$ for very $x'\in M$ and $v\in T_{x'}M'$. Let $x :=\phi(x')$ and $g := z'(x') = z(x)$. To complete the proof of the claim, observe that by \eqref{eqn:def2:dz},
\[ dz'(x')(v) = \Trans{g^{-1}}_*|_g \circ z'_*|_{x'} v = \Trans{g^{-1}}_*|_g \circ z_*|_{x} \circ \phi_*|_{x'} v = (dz \circ\phi_*)(x') v. \]

The proof of Claim~(iv) begins by noting two identities between tangent bundle maps,
\begin{equation}\label{eqn:dfd}\begin{split}
(gz)_* &= \Trans{g}_*\circ z_*, \\
 \Trans{g^{-1}}_*|_{g} \Trans{g}_*|_{e} &= \Trans{g^{-1}g}_*|_e = Id.
\end{split}\end{equation}
Then using the definition of $dz$ in \eqref{eqn:def2:dz} in conjunction with the above identities gives,
\[\begin{split}
d(gz)(V)(x) &= \Trans{z(x)^{-1}g^{-1}}_*|_{gz(x)} (gz)_*|_x V(x) \\
& = \Trans{z(x)^{-1}}_*|_{z(x)} \Trans{g^{-1}}_*|_{g} \Trans{g}_*|_{e} \circ z_*|_x V(x) \\
& = \Trans{z(x)^{-1}}_*|_{z(x)} z_*|_x V(x) = dz(V)(x). \\
\end{split}\]

This completes the proof of Theorem \ref{thm:dz_basic}. \qed

\subsection{Proof of Theorem \ref{thm:gradient}} \label{sec:proof:gradient}

Let $x\in N$ and $w\in T_x N$, then since the function $z$ is constant on $N$, $z_* w = 0$. By \eqref{eqn:def:dz}, this implies that $dz(x)(w) =0$. Thus by the definition of the gradient \eqref{eqn:def:gradient},
\[ 0 = dz(x)(w) = \langle \nabla_{\metric} z, w \rangle_\metric; \quad \forall x\in N, \forall w\in T_x N. \]
This shows that $T_x N \subseteq \ker( \nabla_{\metric} z)$. It remains to be shown that this inclusion is in fact an equality.

By assumption, $x$ is a regular point of $z$, so $z_*(x):T_xM\to T_g \LieG$ has full rank, equal to $d$. By the alternative definition of $dz$ in \eqref{eqn:def:dz}, the kernel of $dz(x)$ is also the kernel of $z_*(x)$ , so it is a $n-d$-dimensional subspace of $T_x M$. Note that $N$ being an $n-d$ dimensional manifold, $T_x N$ is also $n-d$ dimensional. Thus the inclusion $\subseteq$ must be an equality. This completes the proof of Theorem \ref{thm:gradient}. \qed

\subsection{Proof of Theorem \ref{thm:Koop_equiv}} \label{sec:proof:Koop_equiv}

We begin the proof with the following observations,
\begin{equation}\begin{split}\label{eqn:sldfn}
\frac{d}{dt}|_{t=0} \exp^t_\omega\circ z(x) &= \frac{d}{dt}|_{t=0} \Trans{z(x)} \circ \exp(t\omega) = \Trans{z(x)}_*|_e \omega \\
\frac{d}{dt}|_{t=0} z\left( \Phi_V^t(x) \right) &= z_*|_x V(x) \\
\end{split}\end{equation}
Now suppose (i) is true, then taking $z(x)$ as in the second equation in \eqref{def:Koop_G} and taking the derivative $d/dt$ at $t=0$ gives
\[  \Trans{z(x)}_*|_e \omega = \frac{d}{dt}|_{t=0} z\left( \Phi_V^t(x) \right) \stackrel{ \mbox{by \eqref{eqn:sldfn}} }{=} z_*|_x V(x).  \]
The inverse of the term $\Trans{z(x)}_*|_e$ on the left is $\left( \Trans{z(x)}_*|_e \right)^{-1}$, and by the second identity in \eqref{eqn:dfd}, it equals $\Trans{z(x)^{-1}}_*|_{z(x)}$. Thus the above equation can be rewritten as
\[ \omega = \left( \Trans{z(x)}_*|_e \right)^{-1} z_*|_x V(x) = \Trans{z(x)^{-1}}_*|_{z(x)} z_*|_x V(x) = dz(V)(x). \]
The last equality follows from \eqref{eqn:def2:dz}. This completes the first part of the proof. 

Now let (ii) hold. To prove that $z$ is a Koopman eigenfunction, the commutation relation in \eqref{def:Koop_G} has to be proved for all $t\in\real$. Alternatively, one can prove the differential version of that relation, namely, that 
\[ \frac{d}{dt}|_{t=0} \exp^t_\omega\circ z(x) = \frac{d}{dt}|_{t=0} z(\Phi_V^t(x)).\]
This however follows by retracing the proof of the previous part backwards, along with \eqref {eqn:sldfn}. This completes the second and last part of the proof of Theorem \ref{thm:Koop_equiv}. \qed

\subsection{Proof of Theorem \ref{thm:rescale_G}} \label{sec:proof:rescale_G}

To prove Claim~(i), we will show that for every $x\in M$, $dz(\alpha V)(x)$ = $\alpha(x) dz(V)(x)$. Fix a Riemannian metric $\tau$ on $M$, then by Theorem \ref{thm:gradient},
\[ dz(\alpha V)(x) = \langle \nabla_\tau z(x), \alpha(x) V(x) \rangle_\tau = \alpha(x) \langle \nabla_\tau z(x), V(x) \rangle_\tau = \alpha(x) dz(V)(x). \]

For the proof of Claim~(ii),we will begin with the ``if'' part, so let the 1-dimensional subspace $L$ exists as described. Fix a nonzero vector $\omega\in L$. Then for every $x\in M$, by assumption, there is an $\alpha(x)\neq 0$ such that $\alpha(x) dz(V)(x) \equiv \omega$. Then by Claim~(i), the rescaled vector field $\tilde{V} := \alpha V$ satisfies $dz(\tilde{V}) \equiv \omega$. By Theorem \ref{thm:Koop_equiv}, this is equivalent to saying that $z$ is a Koopman eigenfunction of the flow induced by the rescaled vector field $\tilde{V}$, with frequency $\omega$. 

To prove the ``only if'' part, let $\alpha:M\to\real$ be an everywhere non-zero scaling function such that the vector field $\tilde{V}(x)$ = $\alpha(x)V(x)$ has $z$ as a Koopman eigenfunction. Then $z$ will have a frequency $\omega$, for some $\omega\in\LieA$. Then by Theorem \ref{thm:Koop_equiv}, 
\[ dz(V)(x) = \alpha(x)^{-1} \alpha(x) (V z)(x) = \alpha(x)^{-1}dz(\tilde{V} )(x) = \alpha(x)^{-1} \omega; \quad \forall x\in M. \]
Let $L$ be the span of $\omega$. Since $\alpha(x)$ is non-zero, the above equation shows that $dz(V)$ always lies in $L$, proving the claim. This completes the proof of Theorem \ref{thm:rescale_G}. \qed

\subsection{Proof of Corollary \ref{corr:rescale:S1}} \label{sec:proof:rescale:S1} 

In the case $\LieG=S^1$, the dimension $d$ equals $1$ and $\LieA\cong\real$. Substituting $\LieG$, $\LieA$ with $S^1$ and $\real$ respectively, in \eqref{eqn:def2:dz}, gives \eqref{eqn:def2:dz_S1}. This proves Claim~(i). For $S^1$, the complex valued exponential map $r\to e^{ir}$ mapping $\real\to S^1$ is the exponential map between $\LieA$ and $\LieG$. This observation and Theorem \ref{thm:local_lift} stated later will prove Claim~(ii) in more generality. Claims (iii), (iv), (v) are analogous to Theorem \ref{thm:dz_basic}(iii), (iv) and Theorem \ref{thm:rescale_G}(i) respectively. 

To prove Claim~(vi), we begin with the ``only if'' part. So let $\alpha:M\to\real$ be an everywhere non-zero scaling function such that $\tilde{V}(x)$ = $\alpha(x)V(x)$ has $\zeta$ as a Koopman eigenfunction. By \eqref{eqn:def:koop_eigen}, $|\zeta|$ is constant everywhere. This proves condition (i). Let $\zeta$ be rescaled so that $|\zeta|=1$. Then $\zeta$ becomes a map $\zeta:M\to S^1$ and is thus Lie group valued. Then by Theorem \ref{thm:rescale_G}, there is a subspace $L$ such that $d\zeta(V)(x)$ lies in $L$ for every $x\in M$. But since $d=1$, $L=\real=\LieA$. This means that $V$ has a nonzero component along the gradient vector field $\nabla_{\metric}\zeta$. But by Claim~(v), $\nabla_{\metric}\zeta$ is everywhere orthogonal to the foliation induced by $\zeta$, hence $V$ is everywhere transversal to this foliation, proving condition (ii).

The ``if'' part will now be proved. Condition (i) allows us to assume without loss of generality that $|\zeta|\equiv 1$, so that $\zeta:M\to S^1$. By Theorem \ref{thm:gradient}, $d\zeta(V)$ = $\langle \nabla_{\metric} \zeta, V \rangle_{\metric}$, which is non-zero everywhere since $V$ is transversal to the foliation induced by $\zeta$ by condition (ii). Moreover, the span of $d\zeta(V)$ is trivially the 1-dimensional space $L=\real=\LieA$. Thus the condition of Theorem \ref{thm:rescale_G} is met and $V$ can be rescaled to make $\zeta$ a Koopman eigenfunction. This completes the proof of Corollary \ref{corr:rescale:S1}. \qed

\subsection{Proof of Theorem~\ref{thm:submersion}} \label{sec:proof:submersion}

We begin with a lemma that establishes that each non-constant eigenfunction is a submersion.
\begin{lemma}\label{lem:Koop_eigen_subm}
Let $z:M\to S^1$ be a $C^1$ Koopman eigenfunction with eigenfrequency $\omega\neq 0$. Then $z$ is a submersion and its fibres are therefore, codimension-1 submanifolds.
\end{lemma}
\begin{proof} By the Koopman eigenvalue equation \eqref{eqn:def:koop_eigen}, $V(z) = i\omega z$. For any choice of a Riemannian metric $\tau$, 
\[ \omega = dz(V)(x) = \proj_2 z_*(V)(x) = \langle \nabla_\tau z(x), V(x) \rangle_\tau, \quad \forall x\in M.  \]
This implies that $z_*(V)$ is non-zero everywhere on $M$. Thus $z_*$ must be a submersion.
\end{proof}

The proof of Theorem \ref{thm:submersion} will be by contradiction, so suppose that $\pi$ is not a submersion. By Lemma \ref{lem:Koop_eigen_subm}, every Koopman eigenfunction with nonzero eigenfrequency is a submersion, so there is a non-empty maximal subset of \{$z_1,\ldots,z_m$\} which forms a submersion. By renumbering the eigenfunctions, we can assume without loss of generality that this set is $z_1,\ldots,z_k$. So $\pi^{(k)}$ := $(z_1,\ldots,z_k)$ is a submersion, but $\pi^{(k+1)}$ := $(z_1,\ldots,z_{k+1})$ is not. Let $x\in M$ be a singular point for $\pi^{(k+1)}$, so $D\pi^{(k+1)}(x) : T_xM\to T_{y}\TorusD{k+1}$ has rank k, where y:= $\pi^{(k+1)}(x)$. Let $F^{(k)}$ be the foliation by $\pi^{(k)}$. Let $W_k$ be the bundle spanned by the vectors $\nabla_{\metric} z_1, \ldots, \nabla_{\metric} z_k$. Note that $TF^{(k)}$ = $W_k^\bot$ = $\ker\pi^{(k)}$. 

Let $F_{k+1}$ be the foliation induced by $z_{k+1}$. Now note that $D\pi^{(k+1)}(x)$ has rank $k$ iff $\nabla_{\metric} z_{k+1}(x) \in W_k(x)$. In other words
\begin{equation} \label{eqn:sdm38}
    D\pi^{(k+1)}(x) = x \;\Leftrightarrow\; T_{x} F^{(k)} \subset T_{x} F_{k+1}. 
\end{equation}
%
By the invariance of both the foliations $F_{k+1}$ and $F^{(k)}$ under the flow $\Phi^t$, for every $t\in\real$, 
\[ T_{\Phi^t x} F^{(k)} \subset T_{\Phi^t x} F_{k+1}, \quad \forall t\in\real .  \]
%
Thus by \eqref{eqn:sdm38} $\Phi^t x$ is also a singular point for $\pi^{(k+1)}$, for every $t\in\real$. 
Therefore, for every $t\in\real$, $\pi(\Phi^t x)$ = $y+\vec{\omega}t$ is a singular value of the map $\pi$, where $\vec{\omega}$ = $(\omega_1, \ldots, \omega_{k+1})$. However, since the components of $\vec{\omega}$ are rationally independent, the values \{ $y+\vec{\omega}t$ : $t\in\real$ \} form a dense full-measure subset of $\TorusD{k+1}$. This violates Sard's theorem, and completes the proof of Theorem~\ref{thm:submersion}. \qed

\subsection{Proof of Theorem~\ref{thm:local_lift}} \label{sec:proof:local_lift} 

To prove that $\tilde{d}\theta = d(z|u)$, it is equivalent to prove the commutation diagram below. 
\[\begin{tikzcd}[column sep = large]
TU_g \arrow{r}[swap]{\IsoA} & U_g\times\LieA \arrow{r}[swap]{\proj_2} & \LieA \\
TU_x \arrow{u}{z_*} \arrow{r}{\theta_*} &TU_0 \arrow{r}{\IsoB} & U_0\times\LieA \arrow{u}[swap]{\proj_2} 
\end{tikzcd}.\]
In this diagram, the clockwise path from $TU_x$ to $\LieA$ is $d(z|u)$ and the counter-clockwise path is $\tilde{d}\theta$. We will in fact prove the stronger commutation
\[\begin{tikzcd}[column sep = large]
TU_g \arrow{r}[swap]{\IsoA} & U_g\times\LieA \arrow{r}[swap]{\proj_2} & \LieA \\
TU_x \arrow{u}{z_*} \arrow{r}{\theta_*} &TU_0 \arrow{u}{\IsoB} 
\end{tikzcd}.\]
 To prove this, we include the map $E_{g*}:TU_0 \to TU_g$ in the diagram and split the figure into two separate commuting diagrams.
\[
\begin{tikzcd}
TU_g &\ \\
TU_x \arrow{u}{z_*} \arrow{r}{\theta_*} &TU_0 \arrow{ul}[swap]{E_{g*}}
\end{tikzcd}; \quad
\begin{tikzcd}[column sep = large]
TU_g \arrow{r}[swap]{\IsoA} & U_g\times\LieA \\
\ &TU_0 \arrow{ul}{E_{g*}} \arrow{u}{\IsoB} 
\end{tikzcd}.
\]
The first diagram is a direct consequence of the composition relation $z = E_g\circ\theta$. To verify the second diagram, fix $u\in U_0$ and $w\in T_u U_0$. Then,
\[\begin{split} 
\IsoA\circ E_{g*}(u,w) &= \IsoA\circ \left( \Trans{g}_*|_{\exp(u)} \exp_*|_u w \right) \quad \mbox{by \eqref{eqn:def:lift}} , \\
&= \Trans{\left[g\exp(u)\right]^{-1}}_*|_{g\exp(u)} \Trans{g}_*|_{\exp(u)} \exp_*|_u w \quad \mbox{by \eqref{eqn:Lie_bundle}} , \\
&= \Trans{\exp(u)^{-1}}_*|_{\exp(u)} \Trans{g^{-1}}_*|_{g\exp(u)} \Trans{g}_*|_{\exp(u)} \exp_*|_u w , \\
&= \Trans{\exp(u)^{-1}}_*|_{\exp(u)} \exp_*|_u w = \IsoB(u,w) \quad \mbox{by \eqref{eqn:def:Iso2}} . \\
\end{split}\]
This proves the commutation in the second diagram and the first part of the claim is proved. 

The other two claims assume that $\LieG$ is Abelian. If $\LieG$ is Abelian, then $\exp(u+v) = \exp(u)\exp(v)$ for every $u,v\in \LieA$. Using this fact, is well known \citep[e.g.][Sec A.6]{knapp2001representation} that $\LieG$ is isomorphic to $\real^k \times \TorusD{d-k}$ for some $0 \leq k\leq d$. For these spaces, the exponential map linearly wraps the tangent space $\LieA$ at $e\in G$ around $\LieG$, making the $\IsoB$ the same as $\IsoA_{\text{can}}$. We will show how this holds more precisely. 

Note that the first component of $\IsoB$, namely $\proj_1 \IsoB$, is the identity map. So it remains to be shown that $\proj_2 \IsoB$ is the identity map on each fiber. Since $\exp(u+v) = \exp(v)\exp(u)$ for every $u,v\in U_0$, the induced tangent bundle maps from $T_u U_0 \to T_{\exp(u+v)} \LieG$ must be the same, i.e.
\[ \exp_*|_{u+v} = \Trans{\exp(v)}_* |_{\exp(u)} \exp_*|_u \]
Now take $v = -u$ to get
\[ Id = \exp_*|_{0} = \exp_*|_{u-u} = \Trans{\exp(-u)}_* |_{\exp(u)} \exp_*|_u = \Trans{\exp(u)^{-1}}_* |_{\exp(u)} \exp_*|_u = \proj_2 \IsoB. \]
Thus $\IsoB=\IsoA_{\text{can}}$, then $d(z|U) = \tilde{d}\theta = d\theta$, and the theorem is proved. However, we will show separately that in the Abelian case, $d(z|U) = d\theta$, to provide more intuition about these maps.

First, note that the neighborhood $U$ and map $\theta$ depend on the choice of the point $x$. Since $d(z|U)$ is independent of the particular choice of $x$, we will have to show that $d\theta$ evaluated at any point in $U$ is independent of the choice of $x$. To make the dependency on $x$ clearer, $U$ and $\theta$ will henceforth be denoted as $U_x$ and $\theta_x$ respectively.
So let $x'\neq x$ be such that $U_x \cap U_{x'} \neq \emptyset$ . Let $g=z(x)$, $g' = z(x')$. Thus
\[ g \exp \theta_{x} = z|U_{x} , \quad g' \exp \theta_{x'} = z|U_{x'}\]
Since these maps coincide, for every $y\in U_x \cap U_{x'}$,
\[  z(y) = g \exp \theta_{x}(y) = g' \exp \theta_{x'}(y) \quad \Leftrightarrow \quad g^{-1} g' = \exp\left[ \theta_{x'}(y) - \theta_{x}(y) \right] \]
Since $U_x \cap U_{x'}$ is a non-empty open set, and the left-hand side of the second identity above is independent of $y\in U_x \cap U_{x'}$, the Jacobian of the right hand side must be zero. Since by construction, $\theta_{x}$, $\theta_{x'}$ only takes values in $U_0$, and $\exp$ is a local diffeomorphism on $U_0$, 
\[ D \exp\left[ \theta_{x'}(y) - \theta_{x}(y) \right] = 0 \quad \Leftrightarrow \quad D\theta_{x'}(y) = D\theta_{x}(y) . \]
Finally note that since $\theta_x$ and $\theta_{x'}$ are $\LieA$ valued, the exterior derivatives and the Jacobians coincide. Thus,
\[d\theta_{x'}(y) = D\theta_{x'}(y) = D\theta_{x}(y) = d\theta_{x}(y) . \]
This completes the proof of Theorem~\ref{thm:local_lift} \qed

\paragraph{Acknowledgements} This work supported by a grant from the Office of Naval Research, YIP grant N00014-16-1-2649. This work is a part of a project supervised by Dr. Dimitris Giannakis, at the Courant Institute of Technology, New York University, New York. The author is also grateful to the referee for their insightful comments on improving the presentation of the paper.

\bibliographystyle{unsrt_inline_url}
\bibliography{References,Lie_group_ref}

\begin{thebibliography}{10}

\bibitem{Cai_equicont_2020}
F.~Cai.
\newblock \href{http://dx.doi.org/10.1088/1361-6544/ab8a67}{Measure-theoretic
  equicontinuity and rigidity}.
\newblock {\em Nonlinearity}, 33(8):3739, 2020.

\bibitem{DasJim2017_SuperC}
S.~Das and J.~Yorke.
\newblock \href{http://dx.doi.org/10.1088/1361-6544/aa99a0}{Super convergence
  of ergodic averages for quasiperiodic orbits}.
\newblock {\em Nonlinearity}, 31:391, 2018.

\bibitem{DSSY2017_QQ}
S.~Das, Y.~Saiki, E.~Sander, and J.~Yorke.
\newblock \href{http://dx.doi.org/10.1088/1361-6544/aa84c2}{Quantitative
  quasiperiodicity}.
\newblock {\em Nonlinearity}, 30:4111, 2017.

\bibitem{LevnajicMezic2010}
Z~Levnajic and I.~Mezic.
\newblock \href{http://dx.doi.org/10.1063/1.3458896}{Ergodic theory and
  visualization. i. mesochronic plots for visualization of ergodic partition
  and invariant sets}.
\newblock {\em Chaos}, 20:033114, 2010.

\bibitem{KlusEtAl_tensor_2018}
S.~Klus, P.~Gel{\ss}, S.~Peitz, and C.~Sch{\"u}tte.
\newblock \href{http://dx.doi.org/10.1088/1361-6544/aabc8f}{Tensor-based
  dynamic mode decomposition}.
\newblock {\em Nonlinearity}, 31(7):3359, 2018.

\bibitem{LopesEtAl_2008}
A.~Lopes and P.~Thieullen.
\newblock \href{http://dx.doi.org/10.1088/0951-7715/21/10/003}{Eigenfunctions
  of the {L}aplacian and associated {R}uelle operator}.
\newblock {\em Nonlinearity}, 21(10):2239, 2008.

\bibitem{FroylandPadberg09}
G.~Froyland and K.~Padberg.
\newblock \href{http://dx.doi.org/10.1016/j.physd.2009.03.002}{Almost-invariant
  sets and invariant manifolds -- {C}onnecting probabilistic and geometric
  descriptions of coherent structures in flows}.
\newblock {\em Phys. D}, 238:1507--1523, 2009.

\bibitem{FibichKlein2011}
G.~Fibich and M.~Klein.
\newblock \href{http://dx.doi.org/10.1088/0951-7715/24/7/006}{Continuations of
  the nonlinear {S}chr{\"o}dinger equation beyond the singularity}.
\newblock {\em Nonlinearity}, 24(7):2003, 2011.

\bibitem{FengIyer2019}
Y.~Feng and G.~Iyer.
\newblock \href{http://dx.doi.org/10.1088/1361-6544/ab0e56}{Dissipation
  enhancement by mixing}.
\newblock {\em Nonlinearity}, 32(5):1810, 2019.

\bibitem{DasEtAl_Alfaya_ACC_2021}
S.~Das, S.~Mustavee, and S.~Agarwal.
\newblock \href{http://arxiv.org/abs/math/2109.08623}{Uncovering quasi-periodic
  nature of physical systems: A case study of signalized intersections}, 2021.

\bibitem{Murray_partition_2004}
R.~Murray.
\newblock \href{http://dx.doi.org/10.1088/0951-7715/17/5/004}{Optimal partition
  choice for invariant measure approximation for one-dimensional maps}.
\newblock {\em Nonlinearity}, 17(5):1623, 2004.

\bibitem{Gorbatsevich1994}
V.~Gorbatsevich.
\newblock {\em Lie groups and Lie algebras III: Structure of Lie groups and Lie
  algebras}, volume~41.
\newblock Springer Science \& Business Media, 1994.

\bibitem{GallierQuaintance2020}
J.~Gallier and J.~Quaintance.
\newblock Lie groups, lie algebras, and the exponential map.
\newblock In {\em Differential Geometry and Lie Groups}, pages 559--610.
  Springer, 2020.

\bibitem{DasGiannakis_delay_2019}
S.~Das and D.~Giannakis.
\newblock \href{http://dx.doi.org/10.1007/s10955-019-02272-w}{Delay-coordinate
  maps and the spectra of {K}oopman operators}.
\newblock {\em J. Stat. Phys.}, 175:1107–1145, 2019.

\bibitem{DasGiannakis_RKHS_2018}
S.~Das and D.~Giannakis.
\newblock \href{http://dx.doi.org/10.1016/j.acha.2020.05.008}{Koopman spectra
  in reproducing kernel {H}ilbert spaces}.
\newblock {\em Appl. Comput. Harmon. Anal.}, 49:573--607, 2020.

\bibitem{ZhaoGiannakis2016}
Z.~Zhao and D.~Giannakis.
\newblock \href{http://dx.doi.org/10.1088/0951-7715/29/9/2888}{Analog
  forecasting with dynamics-adapted kernels}.
\newblock {\em Nonlinearity}, 29:2888--2939, 2016.

\bibitem{BerryEtAl2015}
T.~Berry, D.~Giannakis, and J.~Harlim.
\newblock \href{http://dx.doi.org/10.1103/PhysRevE.91.032915}{Nonparametric
  forecasting of low-dimensional dynamical systems}.
\newblock {\em Phys. Rev. E.}, 91:032915, 2015.

\bibitem{Giannakis17}
D.~Giannakis.
\newblock \href{http://dx.doi.org/10.1016/j.acha.2017.09.001}{Data-driven
  spectral decomposition and forecasting of ergodic dynamical systems}.
\newblock {\em Appl. Comput. Harmon. Anal.}, 47, 2019.

\bibitem{DGJ_compactV_2018}
D.~Giannakis, S.~Das, and J.~Slawinska.
\newblock \href{http://dx.doi.org/10.1016/j.acha.2021.02.004}{Reproducing
  kernel {H}ilbert space compactification of unitary evolution groups}.
\newblock {\em Appl. Comput. Harmon. Anal.}, 54:75--136, 2021.

\bibitem{GiannakisEtAl2015}
D.~Giannakis, J.~Slawinska, and Z.~Zhao.
\newblock Spatiotemporal feature extraction with data-driven {K}oopman
  operators.
\newblock {\em J. Mach. Learn. Res. Proceedings}, 44:103--115, 2015.

\bibitem{GiannakisDas_tracers_2019}
D.~Giannakis and S.~Das.
\newblock \href{http://dx.doi.org/10.1016/j.physd.2019.132211}{Extraction and
  prediction of coherent patterns in incompressible flows through space-time
  {K}oopman analysis}.
\newblock {\em Phys. D}, 402:132211, 2019.

\bibitem{DasDimitEnik2020}
S.~Das, D.~Giannakis, and E.~Szekely.
\newblock \href{https://arxiv.org/pdf/2004.02172.pdf}{An information-geometric
  approach for feature extraction in ergodic dynamical systems}, 2020.

\bibitem{DSSY_Mes_QuasiP_2016}
S.~Das et~al.
\newblock \href{http://dx.doi.org/10.1209/0295-5075/114/40005}{Measuring
  quasiperiodicity}.
\newblock {\em Europhys. Lett. EPL}, 114:40005--40012, 2016.

\bibitem{Kocergin1973}
A.~Kocergin.
\newblock \href{http://dx.doi.org/10.1070/IM1973v007n06ABEH002087}{Time changes
  in flows and mixing}.
\newblock {\em Izv. Akad. Nauk SSSR Ser. Mat.}, 37:1275–1298, 1973.

\bibitem{Chacon1966}
R.~Chacon.
\newblock \href{http://www.jstor.org/stable/24901813}{Change of velocity in
  flows}.
\newblock {\em Journal of Mathematics and Mechanics}, 16(5):417--431, 1966.

\bibitem{Parry1972}
W.~Parry.
\newblock \href{http://dx.doi.org/10.1112/jlms/s2-5.3.511}{Cocycles and
  velocity changes}.
\newblock {\em J. Lond. Math. Soc. (2)}, s2-5:511--516, 1972.

\bibitem{korda2020optimal}
M.~Korda and I.~Mezi{\'c}.
\newblock \href{http://dx.doi.org/10.1109/TAC.2020.2978039}{Optimal
  construction of {K}oopman eigenfunctions for prediction and control}.
\newblock {\em IEEE Transactions on Automatic Control}, 65(12):5114--5129,
  2020.

\bibitem{shirasaka2020phase}
S.~Shirasaka, W.~Kurebayashi, and H.~Nakao.
\newblock Phase-amplitude reduction of limit cycling systems.
\newblock In {\em The Koopman Operator in Systems and Control}, pages 383--417.
  Springer, 2020.

\bibitem{Gallier2001}
J.~Gallier.
\newblock \href{http://dx.doi.org/978-1-4613-0137-0_14}{Basics of classical
  {L}ie groups: The exponential map, {L}ie groups, and {L}ie algebras}.
\newblock \href{http://dx.doi.org/978-1-4613-0137-0_14}{In {\em Geometric
  Methods and Applications}}, pages 367--414. Springer, 2001.

\bibitem{hall2003lie}
B.~Hall.
\newblock \href{http://dx.doi.org/10.1007/978-0-387-21554-9_2}{Lie algebras and
  the exponential mapping}.
\newblock \href{http://dx.doi.org/10.1007/978-0-387-21554-9_2}{In {\em Lie
  Groups, Lie Algebras, and Representations}}, pages 27--62. Springer, 2003.

\bibitem{bloch2002symmetric}
A.~Bloch, P.~Crouch, J.~Marsden, and T.~Ratiu.
\newblock \href{http://dx.doi.org/10.1088/0951-7715/15/4/316}{The symmetric
  representation of the rigid body equations and their discretization}.
\newblock {\em Nonlinearity}, 15(4):1309, 2002.

\bibitem{Fedorov2005discrete}
Y.~Fedorov and D.~Zenkov.
\newblock \href{http://dx.doi.org/10.1088/0951-7715/18/5/017}{Discrete
  nonholonomic {LL} systems on {L}ie groups}.
\newblock {\em Nonlinearity}, 18(5):2211, 2005.

\bibitem{selig2004lie}
J.~Selig.
\newblock \href{http://dx.doi.org/10.1007/1-4020-2307-3_5}{{L}ie groups and
  {L}ie algebras in robotics}.
\newblock \href{http://dx.doi.org/10.1007/1-4020-2307-3_5}{In {\em
  Computational Noncommutative Algebra and Applications}}, pages 101--125.
  Springer, 2004.

\bibitem{hydon2000symmetry}
P.~Hydon.
\newblock {\em Symmetry methods for differential equations: a beginner's
  guide}.
\newblock 22. Cambridge University Press, 2000.

\bibitem{starrett2007solving}
J.~Starrett.
\newblock \href{http://dx.doi.org/10.1080/00029890.2007.11920470}{Solving
  differential equations by symmetry groups}.
\newblock {\em Amer. Math. Monthly}, 114(9):778--792, 2007.

\bibitem{iserles2000lie}
A.~Iserles et~al.
\newblock \href{http://dx.doi.org/10.1017/S0962492900002154}{Lie-group
  methods}.
\newblock {\em Acta numerica}, 9:215--365, 2000.

\bibitem{hairer2006geometric}
E.~Hairer et~al.
\newblock \href{http://dx.doi.org/10.4171/OWR/2006/14}{Geometric numerical
  integration}.
\newblock {\em Oberwolfach Reports}, 3(1):805--882, 2006.

\bibitem{liu2013lie}
CS. Liun.
\newblock \href{http://dx.doi.org/10.1016/j.aml.2013.01.012}{A lie-group dso
  (n) method for nonlinear dynamical systems}.
\newblock {\em Appl. Math. Lett.}, 26(7):710--717, 2013.

\bibitem{knapp2001representation}
A.~Knapp.
\newblock \href{http://dx.doi.org/10.1515/9781400883974}{{\em Representation
  theory of semisimple groups: an overview based on examples}}.
\newblock Princeton university press, 2001.

\end{thebibliography}
\end{document}